\documentclass[11pt,reqno,oneside]{amsart}
\usepackage[width=6.3in,height=8.8in]{geometry}
\usepackage[english]{babel}
\usepackage{amsmath,amssymb,stmaryrd,bbm,latexsym,color}
\usepackage{pdfsync}
\usepackage{csquotes}
\usepackage{hyperref}
\hypersetup{
     colorlinks   = true,
     citecolor    = blue,
     linkcolor=blue
}
\usepackage{mathrsfs}
\usepackage{upref}
\numberwithin{equation}{section}
\newcommand{\assign}{:=}
\newcommand{\mathd}{\mathrm{d}}

\newcommand{\tmem}[1]{{\em #1\/}}
\newcommand{\tmop}[1]{\ensuremath{\operatorname{#1}}}
\newcommand{\tmtextbf}[1]{{\bfseries{#1}}}

\newtheorem{lemma}{Lemma}
\newtheorem{theorem}{Theorem}

\newcommand{\normtrip}[1]{|\! |\!  | #1 | \! |\!  |}

\newcommand{\cn}{{\mathcal N}}

\newcommand{\ep}{\varepsilon}

\newcommand{\om}{\omega}

\newcommand{\vp}{\varphi}

\newcommand{\lp}{\left(}
\newcommand{\rp}{\right)}

\newtheorem{definition}[theorem]{Definition}

\newtheorem{remark}[theorem]{Remark}

\begin{document}

\title{
  One-dimensional Reflected Rough Differential Equations
}

\author{Aur\'elien Deya}
\address[A. Deya]{Institut Elie Cartan, University of Lorraine
B.P. 239, 54506 Vandoeuvre-l\`es-Nancy, Cedex
France}
\email{Aurelien.Deya@univ-lorraine.fr}

\author{Massimiliano Gubinelli}
\address[M. Gubinelli]{Hausdorff Center for Mathematics \& Institute for Applied Mathematics, University of Bonn,
Endenicher Allee 60,
53115 Bonn, Germany}
\email{gubinelli@iam.uni-bonn.de}

\author{Martina Hofmanov\'a}
\address[M. Hofmanov\'a]{Institute of Mathematics,  Technical University Berlin, Stra\ss e des 17. Juni 136, 10623 Berlin, Germany}
\email{hofmanov@math.tu-berlin.de}

\author{Samy Tindel}
\address[S. Tindel]{Department of Mathematics, Purdue University,
150 N. University Street,  West Lafayette, IN 47907, United States}
\email{stindel@purdue.edu}

\thanks{Financial support by the DFG via Research Unit FOR 2402 is gratefully acknowledged.}
\date{\today}

\begin{abstract}
We prove existence and uniqueness of the solution of a one-dimensional rough differential equation
  driven by a step-2 rough path and reflected at zero. In order to deal with the lack of control of the reflection measure the proof uses some ideas we introduced in a previous work dealing with rough kinetic PDEs~[\href{https://arxiv.org/abs/1604.00437}{arXiv:1604.00437}].
\end{abstract}

\maketitle

\section{Introduction}

In its original formulation~\cite{MR1654527}, Lyons' rough paths theory aimed at the study of the standard differential model 
\begin{equation}\label{std-rde}
dy_t=f(y_t) \, dx_t \quad , \quad y_0=a\in \mathbbm{R}^d \ , \ t\in [0,T] \ ,
\end{equation} 
where  $f:\mathbbm{R}^d \to \mathcal{L}(\mathbbm{R}^N,\mathbbm{R}^d)$ is a smooth enough application and $x,y:[0,T]\to \mathbbm{R}^m$ are (typically non-differentiable) continuous paths. In order to deal with the lack of regularity one has to drop both the classical differential or integral formulation of the problem and turn to a description of the motion on arbitrarily small, but finite scales. Eq.~\eqref{std-rde} can be interpreted as the requirement that increments of $y$ should behave locally as some ``germ'' given by a Taylor-like polynomial approximation of the right hand side.  A rough path  $\mathbbm{X}$ constructed above the irregular signal  $x$ is the given of the appropriate monomials with which such a local approximation is constructed. One of the key results of the rough path theory is that, under appropriate conditions, only one continuous function $y$ can satisfy all these local constraints. In this case we say that  the path $y$ satisfies the \emph{rough differential equations} (RDEs)~\eqref{std-rde}.

\smallskip

While the  approach of Lyons~\cite{MR1654527, lyons-book, MR2314753} stresses more the control theoretic sides of the theory, and in particular the mapping from rough paths over $x$ to rough path over $y$, it has been Davie~\cite{davie} who observed the usefulness of these local expansions. Following Davie's insight, one of the author of the present paper~\cite{controlling-rp} introduced a suitable Banach space where these local expansions can be studied efficiently. The work of Friz and Victoir~\cite{FV} showed also how to systematically generate and analyse the local expansions for~\eqref{std-rde} leading to a very complete theory for RDEs. 

It later turned out that these principles, or at least some adaptation of them, remain valid for other - less standard - differential models, such as delay \cite{rough-delay} or Volterra \cite{rough-volterra} rough equations and homogeneisation of fast/slow systems \cite{kelly}. The basic idea of local coherent expansions as effective description of rough dynamical systems has been developed more recently in numerous PDE settings (see e.g. \cite{paraprod,hairer-rough-pdes,MR3071506}, to mention but a few spin-offs amongst a flourishing literature) leading to the development of the general framework of regularity structures by Hairer \cite{regu-stru}, which allows to handle local expansions of a large class of distributions.  For a recent nice introduction to rough path theory and some applications see~\cite{friz_course_2014}.

\smallskip

This being said, in the vast majority of the situations so far covered by rough paths analysis, and especially in all the above quoted references, the success of the method lies in an essential way on fix-point and contraction mapping methods to establish existence and more importantly uniqueness of the object under consideration. 
Unfortunately, the existence of such a contraction property is not known in the case of the reflected rough differential equation, which we propose to study in this paper. To be more precise, we will focus on the one-dimensional RDE reflected at $0$, which can be described as follows: given a time $T > 0$, a smooth function $f :\mathbbm{R} \to \mathcal{L} (\mathbbm{R}^N ;\mathbbm{R})$ and a $p$-variation $N$-dimensional rough path $\mathbbm{X}$ with $2\leqslant p<3$ (see Definition \ref{defi-rough-path}), find an $\mathbbm{R}_{\geqslant 0}$-valued path $y\in V^p_1 ([0, T])$ and an $\mathbbm{R}_{\geqslant 0}$-valued increasing function (or \enquote{reflection measure}) $m\in  V^1_1 ([0, T])$ that together satisfy
\begin{equation}\label{rrde-informal}
\mathd y (t) = f (y (t)) \mathd \mathbbm{X}_t + \mathd m_t \ , \qquad y_t \mathd m_t=0 \ .
\end{equation}
Thus, the idea morally is to exhibit a path $y$ that somehow follows the dynamics in (\ref{std-rde}), but is also forced to stay positive thanks to the intervention of some regular \enquote{local time} $m$ at $0$. Of course, at this point, it is not exactly clear how to understand the right hand side of (\ref{rrde-informal}), and we shall later give a more specific interpretation of the system, based on rough paths principles (see Definition \ref{def:sol-rrde}).

\smallskip

The stochastic counterpart of (\ref{rrde-informal}), where $\mathbbm{X}$ is a standard $N$-dimensional Brownian motion and the right hand side is interpreted as an It{\^o} integral, has been receiving a lot of attention since the 60's (see e.g. \cite{lions-snitzman,saisho,skorohod,tanaka}), with several successive generalizations regarding the (possibly multidimensional) containment domain of $y$. This Brownian reflected equation has also been investigated more recently through the exhibition of Wong-Zakai-type approximation algorithms \cite{aida-sasaki}.

\smallskip

{When $1\leqslant p<2$, Problem~(\ref{rrde-informal}) can be naturally interpreted and analyzed by means of Young integration techniques. This situation was first considered by Ferrante and Rovira in \cite{rovira} for the $d$-dimensional positive domain $\mathbbm{R}_{\geq 0}^d$, with exhibition of an existence result therein. Using some sharp $p$-variation estimates for the Skorohod map, Falkowski and Slominski~\cite{falkowski-slominski, FS15} have recently provided a full treatment of the Young case (at least when considering reflection on hyperplanes), by proving both existence and uniqueness of the solution.}

\smallskip

{The  more complex rough (or step-2) version of (\ref{rrde-informal}), which somehow extends the Brownian model, has been first considered by Aida in \cite{aida_reflected_2015}, and further analysed by the same author in \cite{aida-2016} for more general multidimensional domains. Nevertheless, in these two references, only  {\it existence} of a solution to (\ref{rrde-informal}) can be established and the {\it uniqueness} issue is left open. The lack of regularity of the Skorohod map clearly appears as the main obstacle towards a uniqueness result in the approach followed in \cite{aida_reflected_2015,aida-2016}.}

\smallskip

Our aim in this study is to complete the above picture in the one-dimensional situation, that is to prove uniqueness of a solution to the problem (\ref{rrde-informal}). Actually, for the reader's convenience, we will also provide a detailed proof of the existence of a solution in this setting, and simplify at the same time some of the arguments used by Aida in \cite{aida_reflected_2015,aida-2016}. The subsequent analysis accordingly offers a thorough - and totally self-contained - proof of well-posedness of the problem (\ref{rrde-informal}). 

\smallskip

The strategy is inspired by the recent results on rough conservation laws \cite{conservation-laws}. Indeed, there is an analogy between \eqref{rrde-informal} and the kinetic formulation of conservation laws where the so-called kinetic measure appears. As for \eqref{rrde-informal}, this measure is unknown and becomes part of the solution which brings significant difficulties, especially in the proof of uniqueness. The latter is then based on a tensorization-type argument, also known as doubling of variables, and subsequent estimation of the difference of two solutions.

In the case of \eqref{rrde-informal}, we put forward a fairly simple proof of uniqueness based on a direct estimation of a difference of two solutions. In particular, in this finite dimensional setting no technical tensorization method is needed.
The existence is then derived from a compactness result, starting from a smooth approximation of the rough path $\mathbbm{X}$. In both cases, the key of the procedure consists in deriving sharp estimates for the remainder term which measures the difference between the (explicit) local expansion and the unknown of the problem. The strategy thus heavily relies on the so-called {\it sewing lemma} at the core of the rough paths machinery (see Lemma \ref{lemma-lambda}). The estimates on the remainder are then converted via a \emph{rough Gronwall lemma} (see Lemma \ref{lemma:basic-rough-gronwall}) into estimates for the unknown (resp. for some function thereof) in order to establish  existence (resp. uniqueness).

\smallskip

The paper is organized along a very simple division. In Section \ref{sec:setting}, we start with a few reminders on the rough paths setting and topologies, which allows us to give a rigorous interpretation of the problem~(\ref{rrde-informal}), as well as the statement of our well-posedness result (Theorem \ref{th:main-wellposedness}). We also introduce the two main technical ingredients of our analysis therein, namely the above-mentioned sewing and Gronwall lemmas, with statements borrowed from \cite{conservation-laws}. Section \ref{sec:uniqueness} is then devoted to the proof of uniqueness, while Section \ref{sec:existence} closes the study with the proof of existence. {Regarding the latter existence issue, we will first provide an exhaustive treatment of the problem in the one-dimensional situation (the main topic of the paper), and then give a few details on possible extensions of our arguments to more general multidimensional domains (Section \ref{subsec:general-existence}).}

\section{Setting and main result}\label{sec:setting}

To settle our analysis, we will need the following notations and definitions taken from rough paths theory. First of all, let us recall
the definition of the increment operator, denoted by $\delta$. If $g$ is a
path defined on $[0, T]$ and $s, t \in [0, T]$ then $\delta g_{st} \assign g_t
- g_s$, if $g$ is a $2$-index map defined on $[0, T]^2$ then $\delta g_{sut}
\assign g_{st} - g_{su} - g_{ut}$. For $g:[0,T]\to E$ and $\varphi: E\to F$ (with $E,F$ two Banach spaces), we will also use the convenient notations
\begin{equation}\label{eq:notation-brackets}
\llbracket \varphi \rrbracket (g)_{st}:=\int_0^1\mathd \tau \, \varphi(g_s+\tau (\delta g)_{st}) \quad , \quad \llbracket \llbracket \varphi \rrbracket \rrbracket (g)_{st}:=\int_0^1\mathd \tau\int_0^\tau \mathd \sigma \, \varphi(g_s+\sigma (\delta g)_{st}) \ .
\end{equation}
Observe in particular that if $\varphi$ is a smooth enough mapping, then
\begin{equation}\label{eq:brackets-phi}
\delta \varphi (g)_{st}=\llbracket \nabla \varphi \rrbracket (g)_{st} \delta g_{st} \quad \text{and} \quad \llbracket \varphi \rrbracket (g)_{st}-\varphi(g_s) =\llbracket \llbracket \nabla \varphi\rrbracket \rrbracket (g)_{st} \ .
\end{equation}

\smallskip

In the sequel, given an interval $I$ we call a {\tmem{control on $I$}} (and
denote it by $\omega$) any superadditive map on $\mathcal{S}_I \assign \{(s, t) \in
I^2 : s \leqslant t\}$, that is, any map $\omega : \mathcal{S}_I \to [0, \infty [$
such that,
\[ \omega (s, u) + \omega (u, t) \leqslant \omega (s, t), \qquad s \leqslant u
   \leqslant t. \]
We will say that a control $\omega$ is {\tmem{regular}} if $\lim_{|t - s| \to 0} \omega
(s, t) = 0$. Also, given a control $\omega$ on a time interval $I = [a, b]$, we
will use the notation $\omega (I) \assign \omega (a, b)$.

\smallskip

Now, given a time interval $I$, a parameter $p > 0$, a Banach space $E$ and a function $g : \mathcal{S}_I\to E$, we define the $p$-variation norm of $g$ as
\[ \| g \|_{\bar{V}^p_2 (I ; E)} \assign \sup_{(t_i) \in \mathcal{P} (I)}
   \left( \sum_i |g_{t_i t_{i + 1}} |^p \right)^{\frac{1}{p}}, \]
where $\mathcal{P} (I)$ denotes the set of all partitions of the interval $I$, and we denote by $\bar{V}^p_2 (I ; E)$ the set of maps $g:\mathcal{S}_I\to E$ for which this quantity is finite.
In this case,
\[ \omega_g (s, t) \assign \| g \|_{\bar{V}^p_2 ([s, t] ; E)}^p \]
defines a control on $I$, and we denote by $V^p_2 (I ; E)$ the set of elements
$g \in \bar{V}^p_2 (I ; E)$ for which $\omega_g$ is regular on $I$. We then denote
by $\bar{V}^p_1 (I ; E)$, resp. $V^p_1 (I ; E)$, the set of paths $g : I  \to E$ such that $\delta g\in \bar{V}^p_2 (I ; E)$, resp. $\delta g \in V^p_2 (I ; E)$.
Finally, we define the space $\bar{V}^p_{2, \text{loc}} (I ; E)$ of maps $g : \mathcal{S}_I \to E$ such that there exists a countable covering $\{I_k \}_k$ of
$I$ satisfying $g \in \bar{V}^p_2 (I_k ; E)$ for every $k$. We write $g \in V^p_{2, \text{loc}} (I ; E)$ if the related controls can be
chosen regular.

\begin{definition}\label{defi-rough-path}
Fix a time $T>0$ and let $N \geqslant 1$, $2\leqslant p<3$. Then we call a continuous $N$-dimensional $p$-variation rough path on $[0,T]$ any pair 
\begin{equation}\label{p-var-rp}
\mathbbm{X}=(\mathbbm{X}^1,\mathbbm{X}^2) \in V^p_2 ([0,T];\mathbbm{R}^{d}) \times V^{p/2}_2 ([0,T]; \mathbbm{R}^{d,d}) 
\end{equation}
that satisfies the relation 
\begin{equation}\label{chen-rela}
  \delta \mathbbm{X}^{2;ij}_{sut}=\mathbbm{X}^{1,i}_{su} \, \mathbbm{X}^{1,j}_{ut} \ , \qquad s<u<t\in [0,T] \ ,\ i,j\in \{1,\ldots,d\} \ .
\end{equation}
Such a rough path $\mathbbm{X}$ is said to be geometric if it can be obtained as the limit, for the $p$-variation topology involved in (\ref{p-var-rp}), of a sequence of smooth rough paths $(\mathbbm{X}^\ep)_{\ep>0}$, that is with $\mathbbm{X}^\ep=(\mathbbm{X}^{\ep,1},\mathbbm{X}^{\ep,2})$ explicitly defined as
$$\mathbbm{X}^{\ep,1,i}_{st}\assign \delta x^{\ep,i}_{st} \ , \qquad \mathbbm{X}^{\ep,2,ij}_{st}\assign \int_s^t \delta x^{\ep,i}_{su} \, \mathd x^{\ep,j}_u \ ,$$
for some smooth path $x^\ep:[0,T] \to \mathbbm{R}^N$. 
\end{definition}

\smallskip

We are now in a position to provide a clear interpretation of the problem (\ref{rrde-informal}).

\begin{definition}\label{def:sol-rrde}
Given a time $T > 0$, a real $a\geq 0$, a differentiable function $f :\mathbbm{R} \to \mathcal{L} (\mathbbm{R}^N ;\mathbbm{R})$ and a $p$-variation $N$-dimensional rough path $\mathbbm{X}$ with $2\leqslant p<3$, a pair $(y, m)\in  V^p_1 ([0, T] ; \mathbbm{R}_{\geqslant 0}) \times V^1_1 ([0, T] ;\mathbbm{R}_{\geqslant 0})$ is said to solve the problem (\ref{rrde-informal}) on $[0,T]$ with initial condition $a$ if there exists a 2-index map $y^\natural \in V^{p/3}_{2, \text{loc}}([0,T];\mathbbm{R})$ such that for all $s,t\in [0,T]$, we have
\begin{equation}
\left\lbrace
\arraycolsep=1pt\def\arraystretch{1.8}
  \begin{array}{c}
   \delta y_{s t} = f_i (y_s) \mathbbm{X}^{1,i}_{s t} + f_{2,ij} (y_s)
    \mathbbm{X}^{2,ij}_{s t} + \delta m_{s t} +y^\natural_{st} \\
    y_0=a \quad \text{and} \quad m_t=\int_0^t \mathbf{1}_{\{ y_u=0\}}  \mathd m_u 
  \end{array} \label{eq:rrde}
	\right. \ ,
\end{equation}
where we have set $f_{2,ij} (\xi) \assign f'_i (\xi)f_j (\xi)$ and $m([0,t]):=m_t$.
\end{definition}

\begin{remark}
Eq.~\eqref{eq:rrde} should be read as the given of a local expansion of the function $y$: it says that around each time point $s$ the function can be approximated by the germ 
$$
t \mapsto    y_s + f (y_s) \mathbbm{X}^1_{s t} + f_2 (y_s)
    \mathbbm{X}^2_{s t} + \delta m_{s t} 
$$
up to terms of order $\omega(s,t)^{p/3}$ where $\omega$ is a control. The term $ \delta m_{s t} $ is characteristic for this reflected problem: the measure $m$ increases only at times $u$ where $y_u=0$ effectively ``kicking'' the path $y$ away from the negative axis. In some sense it can be considered as a Lagrange multiplier enforcing the constraint $y_u \ge 0$ for all $u\in[0,T]$. 
\end{remark}

With this interpretation in hand, our well-posedness result reads as follows:
\begin{theorem} \label{th:main-wellposedness}
Let $T>0$ and $a>0$. If $f\in \mathcal{C}_b^3(\mathbbm{R};\mathcal{L}(\mathbbm{R}^N,\mathbbm{R}))$, that is if $f$ is $3$-time differentiable, bounded with bounded derivatives, and if $\mathbbm{X}$ is a continuous geometric $N$-dimensional $p$-variation rough path, then Problem~\eqref{rrde-informal} admits a unique solution $(y,m)$ on $[0,T]$ with initial condition $a$.
\end{theorem}

Let us conclude this preliminary section with a presentation of the two main technical results that will be used in our analysis, and the proofs of which are elementary and can be found e.g. in \cite{conservation-laws}. 

\begin{lemma}[Sewing lemma]
  \label{lemma-lambda}Fix an interval $I$, a Banach space $E$ and a parameter
  $\zeta > 1$. Consider a map $G : I^3 \to E$ such that $ G \in \mathrm{Im}
  \hspace{0.17em} \delta$ and for every $s < u < t \in I$,
  \[ |G_{s u t} | \leqslant \omega (s, t)^{\zeta}, \]
  for some regular control $\omega$ on $I$. Then there exists a unique element
  $\Lambda G \in V_2^{1 / \zeta} (I ; E)$ such that $\delta (\Lambda G) = G$
  and for every $s < t \in I$,
  \begin{equation}
    | (\Lambda G)_{s t} | \leqslant C_{\zeta} \omega (s, t)^{\zeta}
    \label{contraction},
  \end{equation}
  for some universal constant $C_{\zeta}$.
\end{lemma}

\begin{lemma}[Rough Gronwall Lemma]
  \label{lemma:basic-rough-gronwall}Fix a time horizon $T > 0$ and let $g :
  [0, T] \to [0, \infty)$ be a path such that for some constants $C, L > 0$,
  $\kappa \geqslant 1$ and some controls $\omega_1, \omega_2$ on $[0, T]$ with $\omega_1$ being regular, one has
  \begin{equation}
    \label{eq:a-priori-g} \delta g_{st} \leqslant C (\sup_{0 \leqslant r
    \leqslant t} g_r)  \hspace{0.17em} \omega_1 (s, t)^{\frac{1}{\kappa}} +
    \omega_2 (s, t),
  \end{equation}
  for every $s < t \in [0, T]$ satisfying $\omega_1 (s, t) \leqslant L$. Then
  it holds
\begin{equation*}
\sup_{0 \leqslant t \leqslant T} g_t \leqslant 
2 e^{c_{L,\kappa}\, \omega_1 (0, T)} 
\left\{ g_0 + \sup_{0 \leqslant t \leqslant T} (\omega_2 (0, t) 
e^{- c_{L,\kappa}\,   \omega_1 (0, t) )} \right\},
\end{equation*}
 where $c_{L,\kappa}$ is defined as
  \begin{equation}
    \label{eq:def-c-kappa-L} 
    c_{L,\kappa}  = \sup \left( \frac{1}{L}, \, (2 Ce^2)^{\kappa} \right) .
  \end{equation}

\end{lemma}

%

\section{Uniqueness}\label{sec:uniqueness}

\begin{theorem}\label{thm:uniqueness}
Let $T>0$ and $a>0$. If $f\in \mathcal{C}_b^3(\mathbbm{R};\mathcal{L}(\mathbbm{R}^N,\mathbbm{R}))$ and $\mathbbm{X}$ is an $N$-dimensional $p$-variation rough path, then Problem~\eqref{rrde-informal} admits at most one solution $(y,m)$ on $[0,T]$ with initial condition $a$.
\end{theorem}

\begin{proof}
  Let $(y, \mu)$ and $(z, \nu)$ be two solutions for~(\ref{eq:rrde}).
 Set $Y \assign (y, z) \in V^p_1
  ([0, T] ; \mathbbm{R}^2)$ and with decomposition (\ref{eq:rrde}) in mind, write 
  \begin{equation}
    \delta Y_{s t} = F_i (Y_s) \mathbbm{X}^{1,i}_{s t} + F_{2,ij} (Y_s) \mathbbm{X}^{2,ij}_{s
    t} + \delta M_{st}+ Y^{\natural}_{s t}  \ , \qquad 0 \leqslant s \leqslant t
    \leqslant T. \label{eq:paired-rrde}
  \end{equation}
where we use the shorthands $F_i (Y) \assign (f_i (y), f_i (z))$, $F_{2,ij} (Y) \assign (f_{2,ij} (y), f_{2,ij} (z))$, $M \assign (\mu, \nu)\in V^1_1
  ([0, T] ; \mathbbm{R}^2)$ and $Y^\natural\assign (y^\natural,z^\natural)\in V^{p/3}_{2,\text{loc}}([0, T] ; \mathbbm{R}^2)$. From now non and until the end of the proof, we fix an interval $I\subset [0,T]$ such that $Y^\natural\in V^{p/3}_{2}(I; \mathbbm{R}^2)$ and consider the following controls on $I$:
$$\omega_Y (s, t)\assign \| Y \|_{\bar{V}^p_1 ([s,t] )}^p \ , \qquad  \omega_{Y,\natural} (s, t) \assign \| Y^\natural \|_{\bar{V}^{p/3}_2 ([s,t] )}^{p/3} \ , $$
$$ \omega_\Delta (s, t)\assign \| y-z \|_{\bar{V}^p_1 ([s,t] )}^p\ , \qquad \omega_{\Delta, \natural} (s, t)\assign \| y^\natural-z^\natural \|_{\bar{V}^{p/3}_2 ([s,t] )}^{p/3} \ ,$$
$$ \omega_M (s, t)\assign \| M \|_{\bar{V}^1_1 ([s,t] )}=\| \mu \|_{\bar{V}^1_1 ([s,t] )} +\| \nu\|_{\bar{V}^1_1 ([s,t] )}\ .$$
Without loss of generality, we will assume that $\omega_{\mathbbm{X}}(I)\leqslant 1$, where $\omega_{\mathbbm{X}}$ is a fixed control such that
  \[ | \mathbbm{X}^1_{s, t} | + | \mathbbm{X}^2_{s, t} |^{1 / 2}
     \leqslant \omega_{\mathbbm{X}} (s, t)^{1 / p}, \quad 0 \leqslant s
     \leqslant t \leqslant T \ . \]
		
\smallskip
		
\noindent		
Now, consider a smooth function $\varphi : \mathbbm{R} 
  \rightarrow \mathbbm{R}_{\geqslant 0}$ and set $h (x_1, x_2) \assign \varphi (x_1
  - x_2)$ for all $x_1,x_2 \in \mathbbm{R}$. A direct computation via Taylor expansion, combined with (\ref{eq:paired-rrde}), shows that 
  \begin{equation}
    \delta h (Y)_{s t} = \llbracket \nabla h\rrbracket (Y)_{st}\, \delta Y_{st}=H_i (Y_s) \mathbbm{X}^{1,i}_{s t} + H_{2,ij} (Y_s)
    \mathbbm{X}^{2,ij}_{s t}  + \int_s^t \varphi' (y_u - z_u) (\mathd \mu_u - \mathd \nu_u)+ h^{\natural}_{s t}
    \label{eq:h-dyn}
  \end{equation}
  where $h^{\natural}$ is a map in $V^{p/3}_{2}(I; \mathbbm{R})$, and where we have set, for all $Y=(y,z)\in \mathbbm{R}^2$,
\begin{eqnarray}
H_i (Y) &\assign& \nabla h (Y)F_i(Y)= \varphi' (y - z) (f_i (y) - f_i (z)) \label{eq:def-H}\\
H_{2,ij} (Y) &\assign& \nabla H_i (Y) F_j (Y) \notag\\
&=&\varphi' (y - z) (f_{2,ij} (y) - f_{2,ij} (z)) 
+\varphi'' (y - z) (f_i (y) - f_i (z))(f_j (y) - f_j (z)) \ . \label{eq:def-H-2}
\end{eqnarray}

	\
  
  {\noindent}\tmtextbf{Step 1: A general estimate on $h^{\natural}$.} Given that $h^{\natural}$ is a remainder term, we wish to use the sewing map to
   estimate it. Applying $\delta$ to
  eq.~(\ref{eq:h-dyn}) and using (\ref{chen-rela}), we get, for $0 \leqslant s \leqslant
    u \leqslant t \leqslant T$:
  \begin{eqnarray}
    \delta h^{\natural}_{s u t} &=& 
     \delta H_i (Y)_{su} \mathbbm{X}^{1,i}_{ut}
    - H_{2,ij}(Y_s) \mathbbm{X}^{1,j}_{s u} \mathbbm{X}^{1,i}_{u t} 
    + \delta H_{2,ij} (Y)_{s u}\mathbbm{X}^{2,ij}_{u t} \notag \\ 
    &=& 
    \lp  \delta H_i (Y)_{su}  - H_{2,ij}(Y_s) \mathbbm{X}^{1,j}_{s u} \rp \mathbbm{X}^{1,i}_{ut}
    + \delta H_{2,ij} (Y)_{s u}\mathbbm{X}^{2,ij}_{u t}\label{eq:sewing-one}  .
			\end{eqnarray}
  We need to expand the quantity $\delta H_i (Y)_{su}  - H_{2,ij}(Y_s) \mathbbm{X}^{1,j}_{s u}$ in \eqref{eq:sewing-one}, in order to show that $h^{\natural}$ is suitably small
  and depends in a very precise way on $\varphi$ and on the difference
  $\Delta \assign y - z$. In fact, by Taylor expansion and using~(\ref{eq:paired-rrde}) we get
  \begin{eqnarray*}
\lefteqn{\delta H_i (Y)_{s u} - H_{2,ij} (Y_s) \mathbbm{X}^{1,j}_{s u}
	\ =\ \llbracket \nabla H_i\rrbracket
    (Y)_{s u} \delta Y_{s u} - H_{2,ij} (Y_s) \mathbbm{X}^{1,j}_{s u}}\\
  & = &\llbracket \nabla H_i\rrbracket (Y)_{s u} F_j (Y_s) \mathbbm{X}^{1,j}_{s u} +
     \llbracket \nabla H_i\rrbracket (Y)_{s u} F_{2,jk} (Y_s) \mathbbm{X}^{2,jk}_{s u} +
     \llbracket \nabla H_i\rrbracket (Y)_{s u} Y^{\natural}_{s u}\\
		& &+ \llbracket
     \nabla H_i\rrbracket (Y)_{s u} \delta M_{s u}- H_{2,ij} (Y_s) \mathbbm{X}^{1,j}_{s u} \\
& = &  (\llbracket \nabla H_i\rrbracket (Y)_{s u} - \nabla H_i (Y_s)) F_j (Y_s)
     \mathbbm{X}^{1,j}_{s u} + \llbracket \nabla H_i \rrbracket (Y)_{s u} F_{2,jk} (Y_s)
     \mathbbm{X}^{2,jk}_{s u}\\
		& &+ \llbracket \nabla H_i \rrbracket(Y)_{s u}
     Y^{\natural}_{s u} + \llbracket \nabla H_i \rrbracket(Y)_{s u} \delta M_{s u} \, ,
	\end{eqnarray*}
  since $H_{2,ij} (Y) = \nabla H_i (Y) F_j (Y)$. Plugging this identity back into equation~(\ref{eq:sewing-one})  and neglecting to write down explicitly the time indexes, we end up with:
  \begin{eqnarray}\label{eq:sewing-2}
\delta h^{\natural}
		& = & (\llbracket \nabla H_i\rrbracket (Y)  - \nabla H_i
    (Y)) F_j (Y) \mathbbm{X}^{1,j} \mathbbm{X}^{1,i} + \llbracket \nabla H_i \rrbracket(Y)
     F_{2,jk} (Y) \mathbbm{X}^{2,jk} \mathbbm{X}^{1,i}\nonumber \\
		& &\hspace{2cm}+ \llbracket \nabla H_i \rrbracket(Y)
     Y^{\natural} \mathbbm{X}^{1,i} + \llbracket \nabla H_i\rrbracket (Y)\,  \delta M \, \mathbbm{X}^{1,i} + \delta H_{2,ij} (Y)
    \mathbbm{X}^{2,ij} .
  \end{eqnarray}
Using elementary algebraic manipulations, as well as the relation $H_{2,ij} (Y) = \nabla H_i (Y) F_j (Y)$, we obtain: 
  \begin{eqnarray*}
\lefteqn{  (\llbracket \nabla H_i\rrbracket (Y)_{s u}  - \nabla H_i (Y_s)) F_j (Y_s)}\\  
  &=&  (\llbracket H_{2,ij}\rrbracket (Y)_{s u}  - H_{2,ij} (Y_s)) + (\llbracket \nabla H_i\rrbracket (Y)_{s u}
    F_j (Y_s) - \llbracket H_{2,ij} \rrbracket(Y)_{s u}) \\
    & = & \llbracket \llbracket \nabla H_{2,ij}\rrbracket \rrbracket (Y)_{s u}  \delta Y_{s u}
    - \llbracket \nabla H_i (\cdot) \llbracket \nabla F_j (\cdot) \rrbracket \rrbracket(Y)_{s u}
    \delta Y_{s u} \ ,
  \end{eqnarray*}
where the identity $\llbracket H_{2,ij}\rrbracket (Y)_{s u}  - H_{2,ij} (Y_s)=\llbracket \llbracket \nabla H_{2,ij}\rrbracket \rrbracket (Y)_{s u}  \delta Y_{s u}$ directly stems from \eqref{eq:brackets-phi}, and where we define:
  \[ \llbracket \nabla H_i (\cdot) \llbracket \nabla F_j (\cdot) \rrbracket \rrbracket(Y)_{s
     u} \assign \int_0^1 \nabla H_i (Y_s + \tau \delta Y_{s u}) \int_0^{\tau}
     \nabla F_j (Y_s + \sigma \delta Y_{s u}) \, \mathd \sigma \mathd \tau\ . 
  \]
Therefore,  we can rewrite eq.~(\ref{eq:sewing-2}) as 
  \begin{eqnarray}
\delta h^{\natural}
    & = & \llbracket \llbracket \nabla H_{2,ij}  \rrbracket \rrbracket(Y) \delta
    Y\,\mathbbm{X}^{1,j} \mathbbm{X}^{1,i} - \llbracket \nabla H_i (\cdot) \llbracket \nabla F_j
    (\cdot) \rrbracket \rrbracket(Y) \delta Y\, \mathbbm{X}^{1,j} \mathbbm{X}^{1,i} \nonumber \\
		 &  &\hspace{2cm} + \llbracket\nabla H_i  \rrbracket(Y) F_{2,jk} (Y) \mathbbm{X}^{2,jk} \mathbbm{X}^{1,i} + \llbracket \nabla H_i \rrbracket(Y) Y^{\natural} \mathbbm{X}^{1,i}\nonumber \\
		&  &\hspace{3cm} +\llbracket \nabla H_i \rrbracket(Y) \delta M\, \mathbbm{X}^{1,i} + \llbracket \nabla H_{2,ij}
    \rrbracket(Y) \delta Y\, \mathbbm{X}^{2,ij} \ .\label{expan-del-h-nat}
  \end{eqnarray}
 In order to further evaluate the rhs of this relation in terms of the test function $\vp$, let us write
  explicit expressions for the gradients $\nabla H_i (Y)$ and $\nabla H_{2,ij} (Y)$
  computed at $(a, b) \in \mathbbm{R}^2$:
  \[ \nabla H_i (Y)  (a, b) = \varphi'' (y - z) (f_i (y) - f_i (z)) (a - b) +
     \varphi' (y - z) (f_i' (y) - f_i' (z)) a \]
  \[ + \varphi' (y - z) f_i' (z) (a - b) \]
  and
  \[ \nabla H_{2,ij} (Y) (a, b) = \varphi'' (y - z) (f_{2,ij} (y) - f_{2,ij} (z)) (a - b) +
     \varphi' (y - z) (f_{2,ij}' (y) a - f_{2,ij}' (z) b) \]
  \[ + \varphi''' (y - z) (f_i (y) - f_i (z))(f_j (y) - f_j (z)) (a - b)\]
	\[+ 2
     \varphi'' (y - z) (f_i (y) - f_i (z)) (f_j' (y) a - f_j' (z) b)\ . \]
At this point, consider the quantity
\begin{equation}\label{defi:norm-trip}
\normtrip{\varphi} \assign \sup_{y, z \in \mathbbm{R}} (| \varphi' (y - z) | + | y - z
     | | \varphi'' (y - z) | + | y - z |^2 | \varphi''' (y - z) |) \ .
		\end{equation}
  Then, denoting by $C_f$ any quantity that only depends on $f$, we have for all $1\leq i,j,k\leq N$ and $s<t\in I$,
\begin{eqnarray*}
 | \llbracket \nabla H_i\rrbracket (Y)_{st}  | + | \llbracket \nabla
     H_{2,ij}\rrbracket (Y)_{st}  | &\leqslant& C_{f}\, \normtrip{\varphi} \\
 | \llbracket \nabla H_i \rrbracket(Y)_{st}  F_{2,jk} (Y_s) | &\leqslant &
     C_{f}\, \normtrip{\varphi}\, \| y - z \|_{\infty;[s,t]} \\
	 | \llbracket \nabla H_i (\cdot) \llbracket \nabla F_{j} (\cdot) \rrbracket
     \rrbracket(Y)_{s t} \delta Y_{s t} | &\leqslant&  C_{f}\, \normtrip{\varphi}\,
     (\omega_{\Delta} (s, t)^{1 / p} + \| y - z \|_{\infty;[s,t]} \omega_Y (s, t)^{1
     / p}) \\
  | \llbracket \llbracket \nabla H_{2,ij}  \rrbracket \rrbracket(Y)_{s t} \delta
     Y_{s t} | &\leqslant & C_{f}\, \normtrip{\varphi}\, (\omega_{\Delta} (s, t)^{1 / p} + \|
     y - z \|_{\infty;[s,t]} \omega_Y (s, t)^{1 / p}) \\
 | \llbracket \nabla H_i  \rrbracket(Y)_{s t} Y^{\natural}_{s t} | &\leqslant &  C_{f}\, \normtrip{\varphi}\, (\omega_{\Delta, \natural} (s, t)^{3 / p} + \| y - z
     \|_{\infty;[s,t]} \omega_{Y, \natural} (s, t)^{3 / p})\ . 
			\end{eqnarray*}
Going back to (\ref{expan-del-h-nat}), we get that for all $s<u<t\in I$,
  \[ | \delta h^{\natural}_{s u t} | \leqslant  C_{f}\, \normtrip{\varphi}\,  \big[ \omega_{\star} (s, t) +\omega_{\mathbbm{X}} (s, t)^{2 / 3} \omega_{\Delta} (s, t)^{1/3}+ \omega_{\mathbbm{X}} (s, t)^{1 / 3} \omega_{\Delta,
     \natural} (s, t)\big]^{3 / p} \]
  where $\omega_{\star}$ is the control on $I$ given for every $s< t \in I$ by
$$
	\omega_{\star} (s, t) \assign \omega_M (s, t)^{p / 3}
     \omega_{\mathbbm{X}} (s, t)^{1 / 3}\\
			+\| y - z\|_{\infty;[s,t]}^{p/3}\omega_{\mathbbm{X},Y}(s,t) \ ,
$$
with
		$$\omega_{\mathbbm{X},Y}(s,t)\assign  \omega_{\mathbbm{X}} (s, t) +\omega_Y (s, t)^{1/3} \omega_{\mathbbm{X}} (s, t)^{2
     / 3} +
     \omega_{Y, \natural} (s, t) \ .$$
 We are therefore in a position to apply the sewing lemma and conclude that for all $s<t\in I$, 
  \begin{equation}\label{preliminary-est-h-nat}
	| h^{\natural}_{st} | \leqslant  C_{f,p}\, \normtrip{\varphi}\,  \big[ \omega_{\star} (s, t) +\omega_{\mathbbm{X}} (s, t)^{2 / 3} \omega_{\Delta} (s, t)^{1/3}+ \omega_{\mathbbm{X}} (s, t)^{1 / 3} \omega_{\Delta,\natural} (s, t)\big]^{3 / p} 
		\end{equation}
  for some quantity $C_{f,p}$ that only depends on $f$ and $p$.

 	\
  
  {\noindent}\tmtextbf{Step 2: A first application.} Our aim now is to apply the previous bound to the non-smooth function $\varphi(\xi)=\varphi_0 (\xi) \assign
  | \xi |$. To this end, we will rely on the smooth approximation
  $\varphi_{\varepsilon}$ defined for $\varepsilon > 0$ as
  $\varphi_{\varepsilon} (\xi) = \sqrt{\varepsilon^2 + | \xi |^2}$ for all
  $\xi \in \mathbbm{R}$. Let us denote the associated objects with
  $h_{\varepsilon}, h_{\varepsilon}^{\natural}, H_{\varepsilon,i},
  H_{\varepsilon, 2,ij}$,.... In this case
  \[ | \varphi_{\varepsilon}' (\xi) | \leqslant 1\ , \qquad |
     \varphi_{\varepsilon}'' (\xi) | \leqslant 1 / \sqrt{\varepsilon^2 + | \xi
     |^2} \ , \qquad | \varphi_{\varepsilon}''' (\xi) | \leqslant 3 /
     (\varepsilon^2 + | \xi |^2) \]
and so, with the notation (\ref{defi:norm-trip}), we have the uniform estimate $\normtrip{\varphi_{\varepsilon}} \leqslant 3$. 
Plugging this estimate into \eqref{preliminary-est-h-nat}, we get:
\begin{equation}\label{preliminary-est-h-nat-2}
	| h^{\natural}_{\varepsilon,st} | \leqslant  
	C_{f,p}\,  \big[ \omega_{\star} (s, t) 
	+\omega_{\mathbbm{X}} (s, t)^{2 / 3} \omega_{\Delta} (s, t)^{1/3}
	+ \omega_{\mathbbm{X}} (s, t)^{1 / 3} \omega_{\Delta,\natural} (s, t)\big]^{3 / p}.
		\end{equation}
Furthermore, explicit elementary computations show that
\begin{equation}\label{eq:lim-phi-eps}
\lim_{\ep\to 0} \vp_{\ep} = |\cdot|,
\qquad
\lim_{\ep\to 0} \vp'_{\ep} = \text{sign},
\quad\text{and}\quad
\lim_{\ep\to 0} \vp''_{\ep} = \delta_{0},
\end{equation}
where the first two limits are simple limits of functions, and the last one is understood in the weak sense. Notice that we also use the convention $\text{sign}(0)=0$ above.

With those preliminaries in mind, let us take limits in \eqref{eq:h-dyn}. To begin with,
as $\ep \to 0$, a standard dominated convergence argument and relation \eqref{eq:lim-phi-eps} yield:
\begin{equation}\label{a1}
\int_s^t \varphi'_{\varepsilon} (y_u - z_u) \mathd (\mu_u - \nu_u) \to \int_s^t \text{sign} (y_u - z_u) \mathd (\mu_u - \nu_u).
\end{equation}
In addition, owing to the fact that $y_t\geqslant 0$, $z_t\geqslant 0$, we have
\begin{eqnarray*}
\lefteqn{\int_s^t \text{sign} (y_u - z_u) \mathd (\mu_u - \nu_u)}\\
&=& \int_s^t \mathbf{1}_{\{y_u>z_u \geqslant 0\}}  d\mu_u -\int_s^t \mathbf{1}_{\{z_u>y_u \geqslant 0\}}  d\mu_u-\int_s^t \mathbf{1}_{\{y_u>z_u \geqslant 0\}} d\nu_u+\int_s^t \mathbf{1}_{\{z_u>y_u \geqslant 0\}}  d\nu_u.
\end{eqnarray*}
Hence, using the conditions $\mu_t =\int_0^t \mathbf{1}_{\{y_u=0\}}\mathd \mu_u $, $\nu_t =\int_0^t \mathbf{1}_{\{z_u=0\}}\mathd \nu_u $, we end up with:
\begin{eqnarray}\label{a2}
\lefteqn{\int_s^t \text{sign} (y_u - z_u) \mathd (\mu_u - \nu_u)
=-\Big[\int_s^t \mathbf{1}_{\{z_u>y_u \geqslant 0\}}  d\mu_u+\int_s^t \mathbf{1}_{\{y_u>z_u \geqslant 0\}} d\nu_u \Big]} 
\notag\\
&=& -\Big[\int_s^t \mathbf{1}_{\{z_u>y_u \geqslant 0\}}  d(\mu_u+\nu_u)+\int_s^t \mathbf{1}_{\{y_u>z_u \geqslant 0\}} d(\mu_u+\nu_u) \Big]
\notag\\
&=& - \int_s^t \mathbf{1}_{\{y_u \neq z_u\}} \mathd (\mu_u + \nu_u)\ = \ - \omega_M (s, t) + \int_s^t \mathbf{1}_{\{y_u = z_u\}} \mathd (\mu_u +\nu_u) \ .
\end{eqnarray}
Recall that $H_i$ and $H_{2,ij}$ are defined respectively by \eqref{eq:def-H} and \eqref{eq:def-H-2}. Thanks to \eqref{eq:lim-phi-eps}, it thus clearly holds that
\begin{equation}\label{a3}
 \lim_{\ep\to 0} H_{\varepsilon,i} (Y) = \Psi_i (Y)\ ,
  \quad\text{and}\quad
 \lim_{\ep\to 0} H_{\varepsilon, 2,ij} (Y) 
     = \Psi_{2,ij} (Y) \ , 
\end{equation}
where the limits are simple limits of functions and where we have:
  \[ \Psi_i (Y) \assign \tmop{sgn} (y - z) (f_i (y) - f_i (z)) \ , \qquad \Psi_{2,ij} (Y) \assign \tmop{sgn} (y - z) (f_{2,ij} (y) - f_{2,ij} (z)) \ . \]
Taking relations \eqref{a1}, \eqref{a2} and \eqref{a3} into account, we can now take limits as $\varepsilon \rightarrow 0$ in (\ref{eq:h-dyn}). This ensures the convergence of the quantity $h_{\varepsilon,st}^{\natural}$ to some limit $\Phi^{\natural}_{st}$ (for all $s<t\in I$), and using~(\ref{preliminary-est-h-nat}) we get that the path $\Phi (Y) \assign | y - z |$ satisfies the following equation:
\begin{equation}\label{eq:phi-y}
\delta \Phi (Y)_{st} = \Psi_i (Y_s) \mathbbm{X}^{1,i}_{st} + \Psi_{2,ij} (Y_s) \mathbbm{X}^{2,ij}_{st} - \omega_M (s, t) + \int_s^t \mathbf{1}_{\{y_u = z_u\}}\mathd (\mu_u + \nu_u)+\Phi^{\natural}_{st} \ .
\end{equation}
Moreover, invoking relation \eqref{preliminary-est-h-nat-2},  we have for all $s<t\in I$:
	\begin{equation}\label{boun-phi-nat}
 | \Phi^{\natural}_{s t} | \leqslant C_{f,p}  \big[ \omega_{\star} (s, t)+\omega_{\mathbbm{X}} (s, t)^{2 / 3} \omega_{\Delta} (s, t)^{1/3} + \omega_{\mathbbm{X}} (s, t)^{1 / 3} \omega_{\Delta,
     \natural} (s, t)\big]^{3 / p}\ . 
		\end{equation}
Here and in the sequel, we denote by $C_{f,p}$ any quantity that only depends on $f$ and $p$.		
  
 	\
  
  {\noindent}\tmtextbf{Step 3: Bounds for $\omega_{\Delta}$ and $\omega_{\Delta, \natural}$.} Let us now estimate $\omega_{\Delta}$ and $\omega_{\Delta, \natural}$ in terms of $\omega_{\star}$. To this end, we can first use the fact that the path $\Delta \assign y - z$ is (obviously)
  given by $h (Y)$ with the choice $\varphi (\xi) = \psi (\xi) \assign \xi$.
  In this case $h^{\natural} = y^{\natural} - z^{\natural}$, $\normtrip{\psi}=1$, so that (\ref{preliminary-est-h-nat}) becomes
  \[ | y^{\natural}_{s t} - z^{\natural}_{s t} | \leqslant C_{f,p}\,  \big[ \omega_{\star} (s, t) +\omega_{\mathbbm{X}} (s, t)^{2 / 3} \omega_{\Delta} (s, t)^{1/3}+ \omega_{\mathbbm{X}} (s, t)^{1 / 3} \omega_{\Delta,\natural} (s, t)\big]^{3 / p}  \]
for all $s<t\in I$, and accordingly we have
	\[ \omega_{\Delta,\natural} (s, t) \leqslant C_{f,p}^{(1)}\big[ \omega_{\star} (s, t) +\omega_{\mathbbm{X}} (s, t)^{2 / 3} \omega_{\Delta} (s, t)^{1/3}+ \omega_{\mathbbm{X}} (s, t)^{1 / 3} \omega_{\Delta,\natural} (s, t)\big] \]
	for some fixed constant $C_{f,p}^{(1)}$. As a result, for any interval $I_0\subset I$ satisfying 
	\begin{equation}\label{cond-i-zero}
	C_{f,p}^{(1)}\omega_{\mathbbm{X}} (I_0)^{1 / 3} \leqslant 1/2 \ ,
	\end{equation}
	and for all $s<t\in I_0$, we have
	\begin{equation}\label{first-bou-omega-nat}
 \omega_{\Delta,\natural} (s, t) \leqslant 2C_{f,p}^{(1)}\big[ \omega_{\star} (s, t) +\omega_{\mathbbm{X}} (s, t)^{2 / 3} \omega_{\Delta} (s, t)^{1/3} \big] \ .
\end{equation}
	Besides, going back to the equation satisfied by $\Delta$ (again, take $\varphi(\xi)=\xi$ in (\ref{eq:h-dyn})), we easily obtain that for all $s<t\in I$,
	$$|\delta \Delta_{st}| \leqslant C_{f,p} \big[ \|y-z\|_{\infty;[s,t]}^p \omega_{\mathbbm{X}}(s,t)+\omega_M(s,t)^p+\omega_{\Delta,\natural}(s,t)^3 \big]^{1/p} \ ,$$
so
$$\omega_\Delta(s,t)\leqslant C_{f,p} \big[  \|y-z\|_{\infty;[s,t]}^p \omega_{\mathbbm{X}}(s,t)+\omega_M(s,t)^p+\omega_{\Delta,\natural}(s,t)^3\big]$$
and for any interval $I_0\subset I$ satisfying (\ref{cond-i-zero}), we get by (\ref{first-bou-omega-nat})
$$
\omega_\Delta(s,t)\leqslant C_{f,p}^{(2)}\big[  \|y-z\|_{\infty;[s,t]}^p \omega_{\mathbbm{X}}(s,t)+\omega_M(s,t)^p+\omega_{\star}(s,t)^3+\omega_{\mathbbm{X}} (s, t)^{2} \omega_{\Delta} (s, t)\big]\ ,
$$
for some constant $C_{f,p}^{(2)}$. Finally, for any interval $I_0\subset I$ satisfying both (\ref{cond-i-zero}) and
\begin{equation}\label{cond-i-zero-two}
C_{f,p}^{(2)}\omega_{\mathbbm{X}} (I_0)^{2} \leqslant 1/2 \ ,
\end{equation} 
and for all $s<t\in I_0$, we have
\begin{equation}\label{boun-omega-delta}
\omega_\Delta(s,t)\leqslant 2C_{f,p}^{(2)}\big[  \|y-z\|_{\infty;[s,t]}^p \omega_{\mathbbm{X}}(s,t)+\omega_M(s,t)^p+\omega_{\star}(s,t)^3\big] \ .
\end{equation}
		
 	\
  
  {\noindent}\tmtextbf{Step 4: Conclusion.}  By injecting (\ref{first-bou-omega-nat}) and (\ref{boun-omega-delta}) into (\ref{boun-phi-nat}), we can derive the following assertion: for any interval $I_0\subset I$ satisfying (\ref{cond-i-zero}) and (\ref{cond-i-zero-two}), and all $s<t\in I_0$, it holds that
  \[ | \Phi^{\natural}_{s t} | \leqslant C_{f,p} \big[ \omega_\star(s,t)^{3/p}+\|y-z\|_{\infty;[s,t]} \omega_{\mathbbm{X}}(s,t)^{3/p}+\omega_M(s,t)\omega_{\mathbbm{X}}(s,t)^{2/p}\big]  \ , \]
	which, by the definition of $\omega_\star$, gives 
\[ | \Phi^{\natural}_{s t} | \leqslant C_{f,p} \big[ \|y-z\|_{\infty;[s,t]} \omega_{\mathbbm{X},Y}(s,t)^{3/p}+\omega_M(s,t)\omega_{\mathbbm{X}}(s,t)^{1/p}\big]  \ . \]
Going back to eq.~(\ref{eq:phi-y}) and observing that $\|y-z\|_{\infty;[s,t]} =\sup_{[s,t]} \Phi(Y)$, we obtain that for any such interval $I_0$ and for all $s<t\in I_0$,
  \begin{equation*}
	\delta \Phi(Y)_{st} + \omega_M (s, t)\leqslant C_{f,p}\, \big(\sup_{[s, t]} \Phi (Y)+\omega_M(s,t)\big)\big[\omega_{\mathbbm{X}} (s, t)+\omega_{\mathbbm{X},Y}(s,t)^3\big]^{1/p}+  \int_s^t
     \mathbf{1}_{\{y_u=z_u\}} (\mathd \mu_u + \mathd \nu_u) \ .
		\end{equation*}
We are finally in a position to apply the Rough Gronwall Lemma \ref{lemma:basic-rough-gronwall} with $\om_{1}\assign \om_{\mathbbm{X}}+\om_{\mathbbm{X},Y}^{3}$ and $\om_{2}(s,t)\assign \int_s^t \mathbf{1}_{\{y_u=z_u\}} (\mathd \mu_u + \mathd \nu_u)$, and assert that for every $0\leqslant s<t\leqslant T$, 
  \[ \sup_{[s, t]} \Phi (Y) + \omega_M (s, t) \leqslant C_{f,p,\mathbbm{X},Y}\Big[  \Phi(Y_s)  +
     \int_s^t \mathbf{1}_{\{y_u = z_u\}} (\mathd \mu_u + \mathd \nu_u)\Big] \ , \]
		that is
		  \[ \sup_{r\in [s, t]} |y_r-z_r| + \omega_M (s, t) \leqslant C_{f,p,\mathbbm{X},Y}\Big[ | y_s-z_s| +
     \int_s^t \mathbf{1}_{\{y_u = z_u\}} (\mathd \mu_u + \mathd \nu_u)\Big] \ , \]
		for some constant $C_{f,p,\mathbbm{X},Y}$. 
		
		\smallskip
		
\noindent
Assume now that $[s, t]$ is an interval where $y \neq z$ in $(s, t)$ but $y (s)
  = z (s)$. Then
  \[ \sup_{r\in [s, t]} |y_r-z_r| + \omega_M (s, t) \leqslant 0 \]
  which implies that $\sup_{r\in [s, t]} |y_r-z_r| = 0$ everywhere so we find a
  contradiction and such interval cannot exist. This concludes the proof of
  uniqueness.
\end{proof}

\section{Existence}\label{sec:existence}

\subsection{The one-dimensional case}

\begin{theorem}\label{theo:exi}
Let $T>0$ and $a>0$. If $f\in \mathcal{C}_b^2(\mathbbm{R};\mathcal{L}(\mathbbm{R}^N,\mathbbm{R}))$ and $\mathbbm{X}$ is a geometric $N$-dimensional $p$-variation rough path, then Problem~\eqref{rrde-informal} admits at least one solution $(y,m)$ on $[0,T]$ with initial condition $a$.
\end{theorem}

Just as in \cite{aida_reflected_2015,aida-2016}, our strategy towards existence will appeal to some a priori bound on the measure term of the (approximated) equation. The result more generally applies to the so-called Skorohod problem and it can be read as follows in the one-dimensional case.   

\begin{lemma}\label{lem:skorohod-bound}
Let $g$ be a continuous $\mathbbm{R}$-valued path defined on some interval $I=[\ell_1,\ell_2]$, and consider a solution $(y,m)\in \mathcal{C}(I;\mathbbm{R}_{\geqslant 0}) \times V^1_1(I;\mathbbm{R}_{\geqslant 0})$ of the Skorohod problem associated with $g$ in the domain $\mathbbm{R}_{\geqslant 0}$, that is $(y,m)$ satisfies for all $s<t\in I$
$$\arraycolsep=1pt\def\arraystretch{1.8}
\left\lbrace\begin{array}{c}
\delta y_{st}=\delta g_{st}+\delta m_{st} \ , \\
y_{\ell_1}=g_{\ell_1} \ , \ m_t=\int_0^t \mathbf{1}_{\{y_u=0\}} \mathd m_u
\end{array} \right. \ .
$$
Then for all $s<t\in I$ it holds that 
\begin{equation}\label{skorohod-bound}
\delta m_{st}\leqslant 8\,  \lVert g \rVert_{0,[s,t]} \ ,
\end{equation}
where $\lVert g \rVert_{0,[s,t]}\assign \sup_{s\leqslant u<v\leqslant t} |\delta g_{uv}|$.
\end{lemma}

{The proof of (\ref{skorohod-bound}) can be easily derived from the arguments of the proof of \cite[Lemma 2.3]{aida-sasaki} (namely, the same arguments as those leading to the forthcoming general Lemma \ref{lem:boun-skoro-gene}). Let us provide some details though, not least to give the non-initiated reader an insight on how the specific constraints of the reflecting problem can be exploited.}

\begin{proof}[Proof of Lemma \ref{lem:skorohod-bound}]
For all $s< t \in I$, one has
\begin{eqnarray*}
|\delta y_{st}|^2 \ =\ |\delta g_{st}|^2+|\delta m_{st}|^2+2 \delta g_{st} \delta m_{st}&= & |\delta g_{st}|^2+2\int_s^t \delta m_{su} \, \mathd m_u+2\int_s^t \delta g_{st} \, \mathd m_u \\
&= & |\delta g_{st}|^2+2\int_s^t \delta y_{su} \, \mathd m_u+2\int_s^t \delta g_{ut} \, \mathd m_u\ ,
\end{eqnarray*}
where we have just used the fact that $\delta m_{su} = \delta y_{su} - \delta g_{su}$ for the last identity. Moreover, since $\int_s^t y_u \mathd m_u=\int_s^t y_u \mathbf{1}_{\{y_u=0\}} \mathd m_u=0$ and $y_s \geqslant 0$, we get:
\begin{equation*}
|\delta y_{st}|^2
\leqslant  |\delta g_{st}|^2+2\int_s^t \delta g_{ut} \, \mathd m_u \ .
\end{equation*}
Therefore,
$$|\delta y_{st}|^2 \leqslant \lVert g\rVert^2_{0,[s,t]}+2 \, \lVert g\rVert_{0,[s,t]}\, \delta m_{st}\leqslant 5\, \lVert g\rVert^2_{0,[s,t]}+\frac14\, |\delta m_{st}|^2 \ ,$$
and so $\lVert y\rVert_{0,[s,t]} \leqslant 3\, \lVert g\rVert_{0,[s,t]}+\frac12 \, \delta m_{st}$. Finally,
$$\delta m_{st}\leqslant \lVert y\rVert_{0,[s,t]}+\lVert g\rVert_{0,[s,t]} \leqslant 4 \, \lVert g\rVert_{0,[s,t]} +\frac12\, \delta m_{st} \ , $$
and the result follows.
\end{proof}

\begin{proof}[Proof of Theorem \ref{theo:exi}]
We start from a sequence of smooth rough paths $\mathbbm{X}^{\varepsilon}$
converging to $\mathbbm{X}$ as $\varepsilon \to 0$, in the space of continuous $p$-variation geometric rough
paths. We can then find a regular control $\omega_{\mathbbm{X}}$ such that, for all $s,t\in [0,T]$,
\[ | \mathbbm{X}^1_{s t} | + | \mathbbm{X}^2_{s t} |^{1 / 2} \leqslant
   \omega_{\mathbbm{X}} (s, t)^{1 / p}, \qquad \sup_{\varepsilon >0}\, (|
   \mathbbm{X}^{\varepsilon, 1}_{s t} | + | \mathbbm{X}^{\varepsilon, 2}_{s t}
   |^{1 / 2}) \leqslant \omega_{\mathbbm{X}} (s, t)^{1 / p}\  . \]
For every $\varepsilon >0$, let $X^{\varepsilon}$ be the path which corresponds to
  $\mathbbm{X}^{\varepsilon}$ and consider the solution $y^{\varepsilon}$ to
  reflected ODEs starting from $y_0$:
$$
\left\lbrace
\arraycolsep=1pt\def\arraystretch{1.8}
  \begin{array}{c}
   \mathd y^\varepsilon_t= f (y^{\varepsilon}_t)\, \mathd X^{\varepsilon}_{t}+ \mathd m^{\varepsilon}_{t} \\
    y^{\varepsilon}_0=y_0 \quad \text{and} \quad m^{\varepsilon}_t=\int_0^t \mathbf{1}_{\{ y^{\varepsilon}_u=0\}}  \mathd m^{\varepsilon}_u 
  \end{array} 
	\right. \ .
	$$
{Recall that the existence (and uniqueness) of such a solution is a standard result, based on the Lipschitz regularity of the Skorohod map with respect to the supremum norm.} Then by Taylor expansion it is not difficult to show
  that these solutions correspond to rough solutions $(y^{\varepsilon},
  m^{\varepsilon})$ in the sense of~(\ref{eq:rrde})
  \begin{equation}
    \delta y^{\varepsilon}_{s t} = f_i(y^{\varepsilon}_s)
    \mathbbm{X}^{\varepsilon, 1,i}_{s t} + f_{2,ij} (y^{\varepsilon}_s)
    \mathbbm{X}^{\varepsilon, 2,ij}_{s t} + \delta m_{s t}^{\varepsilon} +
    y^{\varepsilon, \natural}_{s t} \qquad s, t \in [0, T] \label{eq:rde-y}
  \end{equation}
  where $y^{\varepsilon, \natural} \in V_{2}^{p / 3}([0,T];\mathbbm{R})$. Let us set from now on
	$$\omega_{y^\ep}(s,t):=\| y^\ep \|_{\bar{V}^p_1 ([s, t] ; E)}^p \quad , \quad \omega_{\ep,\natural}(s,t):=\| y^{\ep,\natural} \|_{\bar{V}^{p/3}_2 ([s, t] ; E)}^{p/3} \ ,$$
		$$\omega_{m^\ep}(s,t):=\| m^\ep \|_{\bar{V}^{1}_1 ([s, t] ; E)}=\delta m^\ep_{st}=m^\ep([s,t]) \ ,$$
and observe that from eq.~(\ref{eq:rde-y}) we have
	  \begin{equation}
    | \delta y^{\varepsilon}_{s t} | \leqslant C_f (\omega_{\mathbbm{X}} (s,
    t)^{1 / p} + \omega_{\mathbbm{X}} (s, t)^{2 / p}) +
    \omega_{m^{\varepsilon}} (s, t) + \omega_{\varepsilon, \natural} (s, t)^{3
    / p}\ . \label{eq:y-bound}
  \end{equation}
	Here and in the sequel, we denote by $C_f$, resp. $C_{f,p}$, any quantity that only depends on $f$, resp. $(f,p)$. 
  
  \
  
  {\noindent}\tmtextbf{Step 1: Bounds on the approximate solutions.} \ We
  would like to pass to the limit in $\varepsilon$ and obtain solutions of the
  limiting problem. In order to do so we need uniform estimates for
  $y^{\varepsilon, \natural}_{s t}$. They are obtained via an application of
  the sewing map. 

 To this end, one can proceed as in the proof of Theorem \ref{thm:uniqueness}, Step 1. Specifically, we can just replace $Y$ by $y$, $H$ by $f$ and $H_{2}$ by $f_{2}$ in relation \eqref{eq:sewing-one}. We then repeat all the steps up to relation~\eqref{expan-del-h-nat}, which yields the following relation for $\delta y^{\varepsilon, \natural}$ (for more simplicity, we neglect to write down the time indexes explicitly):
  \begin{eqnarray}
\delta y^{\varepsilon, \natural}
    & = & \llbracket \llbracket \nabla f_{2,ij}  \rrbracket \rrbracket(y^{\ep}) \delta y^{\ep}\,\mathbbm{X}^{\varepsilon,1,j} \mathbbm{X}^{\varepsilon,1,i} 
    - \llbracket \nabla f_i (\cdot) \llbracket \nabla f_j (\cdot) \rrbracket \rrbracket(y^{\ep}) \delta y^{\ep}\, \mathbbm{X}^{\varepsilon,1,j} \mathbbm{X}^{\varepsilon,1,i} \nonumber \\
&  &\hspace{2cm} 
+ \llbracket\nabla f_i  \rrbracket(y^{\ep}) f_{2,jk} (y^{\ep}) \mathbbm{X}^{\varepsilon,2,jk} \mathbbm{X}^{\varepsilon,1,i} 
+ \llbracket \nabla f_i \rrbracket(y^{\ep}) y^{\ep,\natural} \mathbbm{X}^{\varepsilon,1,i}\nonumber \\
&  &\hspace{3cm} 
+\llbracket \nabla f_i \rrbracket(y^{\ep}) \delta m^\varepsilon\, \mathbbm{X}^{\varepsilon,1,i} 
+ \llbracket \nabla f_{2,ij} \rrbracket(y^{\ep}) \delta y^{\ep}\, \mathbbm{X}^{\varepsilon,2,ij} \ .\label{expan-del-h-nat-y}
  \end{eqnarray}
Combining this expansion with (\ref{eq:y-bound}), we get, for every interval $I\subset [0,T]$ such that $\omega_{\mathbbm{X}}(I)\leqslant 1$ and all $s<u<t\in I$,
 \begin{eqnarray*}
| \delta y^{\varepsilon, \natural}_{s u t} |
& \leqslant &C_f\big[
     \omega_{y^{\varepsilon}} (s, t)^{1 / p} \omega_{\mathbbm{X}} (s, t)^{2 /
     p}+ ( \omega_{\mathbbm{X}} (s, t)^{2 / p} +
     \omega_{m^{\varepsilon}} (s, t) + \omega_{\varepsilon, \natural} (s,
     t)^{3 / p}) \omega_{\mathbbm{X}} (s, t)^{1 / p} \big]\\
		&\leqslant &C_{f,p} \big[\omega_{\mathbbm{X}} (s, t) +
     \omega_{\mathbbm{X}} (s, t)^{1 / 3}  \omega_{m^{\varepsilon}} (s, t)^{p/3} +\omega_{\mathbbm{X}} (s, t)^{1 / 3} \omega_{\varepsilon, \natural} (s,
     t)\big]^{3/p} \ .
		\end{eqnarray*}
We are therefore in a position to apply the sewing lemma and assert that for every interval $I\subset [0,T]$ such that $\omega_{\mathbbm{X}}(I)\leqslant 1$ and all $s<t\in I$, we have
  \[ |  y^{\varepsilon, \natural}_{s t} | \leqslant C_{f,p} \big[\omega_{\mathbbm{X}} (s, t) +
     \omega_{\mathbbm{X}} (s, t)^{1 / 3}  \omega_{m^{\varepsilon}} (s, t)^{p/3} +\omega_{\mathbbm{X}} (s, t)^{1 / 3} \omega_{\varepsilon, \natural} (s,
     t)\big]^{3/p} \ , \]
		which immediately entails that
		$$\omega_{\ep,\natural}(s,t) \leqslant C^{(1)}_{f,p} \big[\omega_{\mathbbm{X}} (s, t) +
     \omega_{\mathbbm{X}} (s, t)^{1 / 3}  \omega_{m^{\varepsilon}} (s, t)^{p/3} +\omega_{\mathbbm{X}} (I)^{1 / 3} \omega_{\varepsilon, \natural} (s,
     t)\big] \ ,$$
		for some constant $C^{(1)}_{f,p}$. As a result, for every interval $I\subset [0,T]$ such that 
  \begin{equation}
   \omega_{\mathbbm{X}}(I)\leqslant 1 \quad \text{and} \quad C^{(1)}_{f,p} \omega_{\mathbbm{X}} (I)^{1 / 3} \leqslant 1/2 \ ,
    \label{eq:I-small}
  \end{equation}
one has
  \begin{equation}
   \omega_{\ep,\natural}(s,t) \leqslant 2 C^{(1)}_{f,p} \big[\omega_{\mathbbm{X}} (s, t) +
     \omega_{\mathbbm{X}} (s, t)^{1 / 3}  \omega_{m^{\varepsilon}} (s, t)^{p/3} \big] \ , \qquad s<t \in I. \label{eq:ysharp-bound}
  \end{equation}
  
	\
  
  {\noindent}\tmtextbf{Step 2: Control of the approximate measures.} 
	Consider the path $g^\ep:[0,T]\to \mathbbm{R}$ defined as $g^\ep_t:=y^\ep_t-m^\ep_t$, and observe that $(y^\ep,m^\ep)$ is then a solution of the Skorohod problem in $\mathbbm{R}_{\geqslant 0}$ associated with $g^\ep$, in the sense of Lemma \ref{lem:skorohod-bound}. Therefore, by (\ref{skorohod-bound}), it holds that
\begin{equation}\label{appli-skorohod}
\omega_{m^\ep}(s,t)\leqslant 8 \, \| g^\ep \|_{0,[s,t]} \ .
\end{equation}
On the other hand, from eq.~(\ref{eq:rde-y}), we have
$$\delta g^\ep_{st}=f_i (y^{\varepsilon}_s)
     \mathbbm{X}^{\varepsilon, 1,i}_{s t} + f_{2,ij} (y^{\varepsilon}_s)
     \mathbbm{X}^{\varepsilon, 2,ij}_{s t} + y^{\varepsilon, \natural}_{s t} \ ,
     \qquad 0 \leqslant s \leqslant t \leqslant T \, ,$$
		and so
\begin{equation}\label{bound-g-ep}
\| g^\ep \|_{0,[s,t]} \leqslant C_f \big[\omega_{\mathbbm{X}} (s,
     t)^{1 / p} + \omega_{\mathbbm{X}} (s, t)^{2 / p} + \omega_{\ep,\natural} (s,
     t))^{3 / p} \big] \ .
		\end{equation}
		Injecting successively (\ref{bound-g-ep}) and (\ref{eq:ysharp-bound}) into (\ref{appli-skorohod}) yields that for every interval $I$ satisfying the conditions in (\ref{eq:I-small}) and every $s<t\in I$,
		$$\omega_{m^\ep}(s,t)\leqslant C^{(2)}_{f,p} \big[ \omega_{\mathbbm{X}}(s,t)^{1/p}+\omega_{\mathbbm{X}}(I)^{1/p} \omega_{m^\ep}(s,t) \big] \ ,$$
		for some constant $C^{(2)}_{f,p}$, and so, if we assume in addition that
\begin{equation}\label{cond-i-two}
C^{(2)}_{f,p}\omega_{\mathbbm{X}}(I)^{1/p} \leqslant 1/2 \ ,
\end{equation}
		we obtain 
\begin{equation}\label{bou-m-ep-loc}
\omega_{m^\ep}(s,t)\leqslant 2C^{(2)}_{f,p} \omega_{\mathbbm{X}}(s,t)^{1/p} \ , \qquad s<t\in I \ .
\end{equation}
From here we can easily conclude that
\begin{equation}\label{eq:measure-bounds}
\omega_{m^\ep}([0,T])\leqslant C_{f,p,\mathbbm{X}} 
\end{equation}
for some quantity $C_{f,p,\mathbbm{X}}$ independent from $\ep$.

	\

  {\noindent}\tmtextbf{Step 3: Passage to the limit for the measure.} With
  all the bounds in place we can now pass to the limit as $\varepsilon
  \rightarrow 0$ via subsequences. We start with the measure.
  Using~(\ref{eq:measure-bounds}) we can assert that there exists a weakly convergent subsequence
  of measures $(m^{\varepsilon (k)})_{k \geqslant 1}$ on $[0,T]$, and we will denote by $m$
  their limit. Then it holds that
  \begin{equation}
    m ([0, t]) = \lim_k m^{\varepsilon (k)} ([0, t]) \qquad t \in
    \mathfrak{C} \label{eq:convergence-mu}
  \end{equation}
  where $\mathfrak{C} \subseteq [0, T]$ is the (dense) set of continuity points of the function $t\mapsto m([0,t])$. Now consider any interval $I$ satisfying both the conditions in (\ref{eq:I-small}) and in (\ref{cond-i-two}), and for $s<t\in I$, introduce a sequence $s_\ell$, resp. $t_\ell$, of points in $\mathfrak{C}$ decreasing to $s$, resp. increasing to $t$, and such that $s_k<t_k$. Using (\ref{bou-m-ep-loc}), we have
	$$m(]s_\ell,t_\ell])=\lim_{k} m^{\ep(k)}(]s_\ell,t_\ell]) \leqslant C_{f,p} \, \omega_{\mathbbm{X}}(s,t)^{1/p} \ ,$$
	and so $m([s,t])\leqslant C_{f,p} \, \omega_{\mathbbm{X}}(s,t)^{1/p} $, which proves that the function $m_t:=m([0,t[)$ is continuous and accordingly that $m\in V^1_1([0,T];\mathbbm{R}_{\geqslant 0})$, as expected.

  \
  
  {\noindent}\tmtextbf{Step 4: Passage to the limit for the path.} Consider
  the subsequence $(y^{\varepsilon (k)}, m^{\varepsilon (k)})_k$ as defined in
  the previous step. Using \eqref{eq:y-bound} we have, for all $s, t \in [0, T]$,
  \[ \limsup_k | \delta y^{\varepsilon (k)}_{s t} | \leqslant C_f
     (\omega_{\mathbbm{X}} (s, t)^{1 / p} + \omega_{\mathbbm{X}} (s, t)^{2 /
     p}) + \omega_m (s, t) + \limsup_k \omega_{\varepsilon (k), \natural} (s,
     t)^{3 / p}\ ,\]
  and for every interval $I$ satisfying both the conditions in (\ref{eq:I-small}) and in (\ref{cond-i-two}) (we
  denote $\mathcal{J}$ the family of such intervals), we have
 $$\limsup_k \omega_{\varepsilon (k), \natural} (s, t)\leqslant  C_{f,p} \big[ \omega_{\mathbbm{X}} (s, t) + \omega_{\mathbbm{X}}
     (s, t)^{1 / 3}\omega_m (s, t)^{p/3}\big]\ ,
     \quad s< t \in I  \ .$$
  From this bound we can choose a further subsequence, still called
  $(y^{\varepsilon (k)}, m^{\varepsilon (k)})_k$ so that $y^{\varepsilon (k)}
  \rightarrow y$ in $C ([0, T];\mathbbm{R}_{\geqslant 0})$. It is easy now to pass to the limit in eq.~(\ref{eq:rde-y}) and
  conclude that there exists a map $y^\natural :\Delta_{[0,T]} \to \mathbbm{R}$ such that
  \[ \delta y = f_i (y) \mathbbm{X}^{1,i} + f_{2,ij} (y) \mathbbm{X}^{2,ij} + \delta m +
     y^{\natural} \ , \]
 and
  \[ |  y^{\natural}_{s t} | \leqslant C_{f,p} \big[ \omega_{\mathbbm{X}} (s, t) + \omega_{\mathbbm{X}}
     (s, t)^{1 / 3}\omega_m (s, t)^{p/3}\big]^{3 / p},
     \qquad s< t \in I \in \mathcal{J}. \]
The fact that $m_t=\int_0^t \mathbf{1}_{\{y_u=0\}} \, \mathd m_u$ (for all $t$) follows immediately from the relation $m^\ep_t=\int_0^t \mathbf{1}_{\{y^\ep_u=0\}} \, \mathd m^\ep_u$, and finally the pair $(y,m)$ does define a solution to the RRDE~(\ref{eq:rrde}).
\end{proof}

\smallskip

\subsection{Generalization to multidimensional domains}\label{subsec:general-existence}
We conclude this study with a few details on possible extensions of the previous arguments (towards existence) to more general multidimensional domains. Together, these results will thus offer a simplification of some of the arguments and topologies used in \cite{aida_reflected_2015,aida-2016}.

\smallskip

Let us first extend Definition \ref{def:sol-rrde} of a reflected rough solution to more general settings, along the classical approach of the reflected problem. Let $D \subset \mathbbm{R}^d$ be a connected domain and for every $x\in \partial D$, denote by $\mathcal{N}_x$ the set of
inward unit normal vectors at $x$, that is 
$$\mathcal{N}_x \assign \cup_{r>0} \mathcal{N}_{x,r} \  , \quad \mathcal{N}_{x,r} \assign
\{
n \in \mathbbm{R}^d: \,  |n| = 1, \, B(x-rn,r) \cap D = \emptyset
\}$$
where $B(z,r)\assign \{y \in  \mathbbm{R}^d: \ |y-z| < r\}$, for $z\in \mathbbm{R}^d$ and $r > 0$.

\begin{definition}\label{def:sol-rrde-multidim}
Given a time $T > 0$, an element $a\in D$, a differentiable function $f :\mathbbm{R}^d \to \mathcal{L} (\mathbbm{R}^N ;\mathbbm{R}^d)$ and a $p$-variation $N$-dimensional rough path $\mathbbm{X}$ with $2\leqslant p<3$, a pair $(y, m)\in  V^p_1 ([0, T] ; D) \times V^1_1 ([0, T] ;\mathbbm{R}^d)$ is said to solve the reflected rough equation in $D$ with initial condition $a$ if there exists a 2-index map $y^\natural \in V^{p/3}_{2, \text{loc}}([0,T];\mathbbm{R}^d)$ such that for all $s,t\in [0,T]$, we have
\begin{equation}
\left\lbrace
\arraycolsep=1pt\def\arraystretch{1.8}
  \begin{array}{c}
   \delta y_{s t} = f_i (y_s) \mathbbm{X}^{1,i}_{s t} + f_{2,ij} (y_s)
    \mathbbm{X}^{2,ij}_{s t} + \delta m_{s t} +y^\natural_{st} \\
    y_0=a \quad \text{and} \quad m_t=\int_0^t \mathbf{1}_{\{ y_u\in \partial D\}}n_{y_u}  \mathd |m|_u 
  \end{array} \label{eq:rrde-multidim}
	\right. \ ,
\end{equation}
where we have set $f_{2,ij} (\xi) \assign \nabla f_i (\xi)f_j (\xi)$, $|m|_t:=\|m\|_{\bar{V}^1_1 ([0, t] ; \mathbbm{R}^d)}$ and for each $y\in \partial D$, $n_y \in \cn_{y}$.
\end{definition}

\smallskip

The existence of a solution for (\ref{eq:rrde-multidim}) can actually be derived from the same arguments as in the one-dimensional situation. The only step of the procedure needing for a revision is the so-called Step 2, since it involves the a priori bound (\ref{skorohod-bound}) which is specific to the one-dimensional Skorohod problem. To this end, we shall exploit the following (sophisticated) substitute, borrowed from  \cite[Lemma 2.2]{aida_reflected_2015}.

\begin{lemma}\label{lem:boun-skoro-gene}
Let $D \subset \mathbbm{R}^d$ be connected domain that satisfies the two following assumptions:

\smallskip

\noindent
{\em\textbf{(A)}} There exists a constant $r_0 > 0$ such that $\cn_x=\cn_{x,r_0}\neq \emptyset$ for any $x\in \partial D$  ;

\smallskip

\noindent
{\em\textbf{(B)}} There exist constants $\delta_0 > 0$ and $\beta \geqslant 1$ satisfying: for every $x\in \partial D$, there exists a unit vector $l_x$ such that $\langle l_x ,n\rangle \geqslant 1/\beta$
for every $n \in \cup_{y\in B(x,\delta_0)\cap \partial D} \cn_y$  . 

\smallskip

\noindent
Let $g\in V_1^p(I;\mathbbm{R}^d)$, for some interval $I=[\ell_1,\ell_2]$, such that $g_{\ell_1}\in D$, and consider a solution $(y,m)\in \mathcal{C}(I;D) \times V^1_1(I;\mathbbm{R}^d)$ of the Skorohod problem associated with $g$ in the domain $D$, that is $(y,m)$ satisfies for all $s<t\in I$
$$\arraycolsep=1pt\def\arraystretch{1.8}
\left\lbrace\begin{array}{c}
\delta y_{st}=\delta g_{st}+\delta m_{st} \ , \\
y_{\ell_1}=g_{\ell_1} \ , \ m_t=\int_0^t \mathbf{1}_{\{y_u=0\}} n_{y_u}\mathd |m|_u
\end{array} \right. \ ,
$$
where $|m|_t:=\|m\|_{\bar{V}^1_1 ([0, t] ; \mathbbm{R}^d)}$ and for each $y\in \partial D$, $n_y \in \cn_{y}$.
Then for all $s<t\in I$ it holds that 
\begin{equation}\label{skorohod-bound-general}
\| m\|_{V_1^1([s,t])} \leqslant C_1  [e^{p C_2 (1 + \|
     g \|_{0, [s, t]})} \|g\|_{\bar{V}_1^p([s,t])} + 1]^{}
     (e^{C_2 (1 + \| g \|_{0, [s, t]})} + 1) \| g
     \|_{0, [s, t]}\ ,
\end{equation}
where $C_1, C_2$ are constants depending only on the domain and $\lVert g \rVert_{0,[s,t]}\assign \sup_{s\leqslant u<v\leqslant t} |\delta g_{uv}|$.
\end{lemma}

\smallskip

\begin{theorem}\label{theo:existence-general-domains}
Let $D\subset \mathbbm{R}^d$ be a connected domain satisfying Conditions {\em\textbf{(A)}} and {\em\textbf{(B)}} of Lemma \ref{lem:boun-skoro-gene}. Then there exists at least one solution $(y,m)$ to the reflection problem \eqref{eq:rrde-multidim} in $D$.
\end{theorem}

\begin{remark}
Of course, Theorem \ref{theo:exi} can retrospectively be obtained as a particular application of Theorem \ref{theo:existence-general-domains}. Nevertheless, we have found it important, for pedagogical reasons, to first provide a full and self-contained treatment of the one-dimensional situation.   
\end{remark}

\begin{proof}
As mentionned above, and apart from minor changes of notation due to the vectorial character of the equation, Steps 1, 3 and 4 of the proof of Theorem \ref{theo:exi} can  be readily transposed to this setting, and thus we only need to focus on the extension of Step 2.

\smallskip

In fact,  with the same notations as in the one-dimensional proof and considering only those intervals $I = [s_0,t_0]$ satisfying the two conditions in (\ref{eq:I-small}), we have by (\ref{skorohod-bound-general}), (\ref{bound-g-ep}) and (\ref{eq:ysharp-bound}) that for all $s<t \in I$,
  \begin{equation}
    \omega_{m^{\varepsilon}} (s, t) \leqslant \Psi (\omega_{g^{\varepsilon}, \varepsilon}
    (s, t)) \leqslant \Psi (C_{f,p}  (\omega_{\mathbbm{X}} (s, t) +
    \omega_{\mathbbm{X}} (s, t) \omega_{m^{\varepsilon}} (s, t)^p)),
    \label{eq:bound-meas}
  \end{equation}
  where
  \[ \Psi (\lambda) \assign C_1  [e^{p C_2 (1 + \lambda^{1 / p})} \lambda +
     1]^{} (e^{C_2 (1 + \lambda^{1 / p})} + 1) \lambda^{1 / p}  \]
	and $C_{f,p}$ is a fixed constant. Eq.~(\ref{eq:bound-meas}) implies in particular that the control
  $\omega_{m^{\varepsilon}}$ is regular  if $\omega_{\mathbbm{X}}$ is
  regular, which is our case. Let $G_I$ be the function
  \[ G_I (\lambda) \assign \Psi (C_{f,p}  (1 + \omega_{\mathbbm{X}} (I)
     \lambda^p)) . \]
  By choosing $t_0$ near to $s_0$ we can have both (\ref{eq:I-small}) and $G_I (3 G_I (0)) \leqslant 2 G_I
  (0)$, since $\omega_{\mathbbm{X}} (I) \rightarrow 0$ as $t_0 \downarrow s_0$.
  This choice of $t_0$ depends only on $\omega_{\mathbbm{X}}$ and $G_I(0)$ (which is actually independent of $I$). Now
  eq.~(\ref{eq:bound-meas}) implies also that
  \[ \omega_{m^{\varepsilon}} (s_0, t) \leqslant G_I (\omega_{m^{\varepsilon}}
     (s_0, t))\ , \qquad t \in I \ . \]
  We want to establish that $\omega_{m^{\varepsilon}} (I) \leqslant 2 G_I
  (0)$ and to this end we can apply the method of continuity. Let $\mathcal{A} \subseteq I$ be
  the set of $t \in I$ such that the property $\omega_{m^{\varepsilon}} (s_0,
  t) \leqslant 2 G_I (0)$ is true. Note that $[s_0, s_0 + \delta] \subseteq
  \mathcal{A}$ for $\delta$ small enough by the continuity of the control
  $\omega_{m^{\varepsilon}}$. Moreover $\mathcal{A}$ is closed in $I$ since if
  $(t_n)_n \subseteq \mathcal{A}$ is a sequence converging to $t_{\ast}$ then,
  again by regularity of $\omega_{m^{\varepsilon}}$  we have
  $\omega_{m^{\varepsilon}} (s_0, t_{\ast}) = \lim_n \omega_{m^{\varepsilon}}
  (s_0, t_n) \leqslant 2 G_I (0)$. Finally $\mathcal{A}$ is also open in $I$
  since if $t_{\ast} \in \mathcal{A}$ then for $\delta$ small enough
  $\omega_{m^{\varepsilon}} (s_0, t) \leqslant 3 G_I (0)$ for all $t \in
  (t_{\ast} - \delta, t_{\ast} + \delta) \cap I$. But then our
  choice of $I$ guarantee that
  \[ \omega_{m^{\varepsilon}} (s_0, t) \leqslant G_I
     (\omega_{m^{\varepsilon}} (s_0, t)) \leqslant G_I (3 G_I (0)) \leqslant 2
     G_I (0) \ , \qquad t \in (t_{\ast} - \delta, t_{\ast} + \delta)
     \cap I \ , \]
  from which we see that $(t_{\ast} - \delta, t_{\ast} + \delta) \cap I
  \subseteq \mathcal{A}$ and that $\mathcal{A}$ is open in $I$. We can then
  conclude that $\mathcal{A}= I$, namely that $\omega_{m^{\varepsilon}} (I)
  \leqslant 2 G_I (0)$. 
  Now we can reason in this way for any nonempty interval $I_{t, \delta} = (t
  - \delta, t + \delta) \cap [0, T]$ by choosing $\delta = \delta (t) > 0$
  small enough to satisfy our conditions. In this way we construct an open
  covering $\cup_t I_{t, \delta (t)}$ of $[0, T]$ from which we can extract a
  finite covering $(I_k)_k$ independent of $\varepsilon$ and such that
$$
\omega_{m^{\varepsilon}} (I_k) \leqslant 2 G_I (0)
$$
for all $I_k$ in the covering. This bound provides us with the expected substitute for (\ref{eq:measure-bounds}), and we can then follow Steps 3 and 4 of the proof of Theorem \ref{theo:exi} to get the conclusion. 
\end{proof}

\end{document}